\documentclass[a4paper,12pt,fleqn]{article}
\usepackage[latin1]{inputenc}  \usepackage[T1]{fontenc}
\usepackage{lscape}
\usepackage{tabularx}     
\usepackage{pstricks}    
\usepackage{rotating}
\usepackage{amsfonts,amsmath,amsthm,amssymb,dsfont}
\usepackage{marginnote}
\usepackage{rotate}
\usepackage{pgfplots}
\usepackage{tikz}
\usepackage{tikz-cd}
\usepackage[arrow,matrix,curve]{xy} 	
\usepackage{xcolor}
\title{Sequences with almost Poissonian Pair Correlations}
\author{Christian Weiß \and Thomas Skill}
\date{\today}

\newtheorem{thm}{Theorem}
\newtheorem{defi}[thm]{Definition}

\newtheorem{lem}[thm]{Lemma}
\newtheorem{prop}[thm]{Proposition}
\newtheorem{cor}[thm]{Corollary}

\newcommand{\RR}{{\mathbb{R}}}

\newcommand{\NN}{{\mathbb{N}}}

\newcommand{\norm}[1]{\left\lVert#1\right\rVert}

\makeindex

\begin{document} 

\maketitle

\thanks{\textit{This paper is dedicated to the Ruhr West University of Applied Sciences on its 10th anniversary}}

\begin{abstract} Although a generic uniformly distributed sequence has Poissonian pair correlations, only one explicit example has been found up to now. Additionally, it is even known that many classes of uniformly distributed sequences, like van der Corput sequences, Kronecker sequences and LS sequences, do not have Poissonian pair correlations. In this paper, we show that van der Corput sequences and the Kronecker sequence for the golden mean are as close to having Poissonian pair correlations as possible: they both have $\alpha$-pair correlations for all $0 < \alpha < 1$ but not for $\alpha = 1$ which corresponds to Poissonian pair correlations.
\end{abstract}


\section{Introduction}

The local and global properties of uniformly distributed sequences $(x_n)$ in $[0,1)$ have been extensively studied for many years. Among these properties Poissonian pair correlations have gained lots of attention in recent time. Intuitively speaking, Poissonian pair correlations mean that the number of small distances between points of the sequence has the expected order of magnitude, see \cite{LS19}. More precisely, the norm of a point $x \in \RR$ is defined by $\norm{x} := \min(x-\lfloor x \rfloor, 1-(x-\lfloor x \rfloor))$, where $\lfloor x \rfloor$ denotes the Gau\ss{} bracket. We say that a sequence $(x_n)$ in $[0,1)$ has \textbf{Poissonian pair correlations} if 
$$\lim_{N \to \infty} \frac{1}{N} \# \left\{ 1 \leq l \neq m \leq N \ : \ \norm{x_l - x_m} \leq \frac{s}{N} \right\} = 2s.$$
Poissonian pair correlations are a generic property of uniformly distributed sequences, i.e. an i.i.d. random sequence (sampled from the uniform distribution in $[0,1)$) almost surely is Possonian. The opposite is true as well.
\begin{thm}[Aistleitner-Lachmann-Pausinger, \cite{ALP18}; Grepstad-Larcher, \cite{GL17}] Let $(x_n)$ be a sequence in [0,1) and assume that
	$$\lim_{N \to \infty} \frac{1}{N} \# \left\{ 1 \leq l \neq m \leq N \ : \ \norm{x_l - x_m} \leq \frac{s}{N} \right\} = 2s$$
	holds for all $s>0$. Then $(x_n)$ is uniformly distributed. 
\end{thm}
Although Poissonian pair correlations are a generic property of uniformly distributed sequences, it seems to be a very hard task to find explicit examples. The only known positive result is $(x_n) = \left\{\sqrt{n}\right\}$, see \cite{BMV15}, where $\left\{ z \right\} := z - \lfloor z \rfloor$ denotes the fractional part of $z$. Indeed, many canonical candidates of uniformly distributed sequences are known to fail having Poissonian pair correlations, see \cite{BBC19}, \cite{LS18}, \cite{PS18}. In this paper, we work with a definition that allows us to measure how far a sequence is from having Poissonian pair correlations. It is going back \cite{NP07} and we have learned about it from \cite{Ste18}.
\begin{defi}
	A sequence $(x_n)$ in $[0,1)$ has \textbf{$\alpha$-pair correlations} for $0 < \alpha \leq 1$ if
	$$F_N^\alpha(s) := \lim_{N \to \infty} \frac{1}{N^{2-\alpha}} \# \left\{ 1 \leq l \neq m \leq N \ : \ \norm{x_l - x_m} \leq \frac{s}{N^\alpha} \right\} = 2s.$$
\end{defi} 
In fact, Steinerberger showed that also $\alpha$-pair correlations imply uniform distribution.
\begin{thm}[Steinerberger, \cite{Ste18}] Let $(x_n)_{n \in \NN} \in [0,1)$ be such that for every $s > 0$ we have that
	$$\lim_{N \to \infty} F_N^\alpha(s) = 2s$$
	then $(x_n)_{n \in \NN}$ is uniformly distributed in $[0,1)$.
\end{thm}
 
It is an important and easy-to-prove fact that the greater $\alpha$ is, the harder it is that a sequence has $\alpha$-pair correlations.
\begin{prop} \label{prop:ordering} Let $1 \geq \alpha_1 > \alpha_2 > 0$ and let $(x_n)$ be a sequence having $\alpha_1$-pair correlation. Then $(x_n)$ also has $\alpha_2$-pair correlations.
\end{prop}
\begin{proof} Since $(x_n)$ has $\alpha_1$-pair correlations, it follows that
	\begin{align*}
	\# \left\{ 1 \leq l \neq m \right.& \leq N \  : \left. \ \norm{x_l - x_m} \leq \frac{s}{N^{\alpha_2}} \right\}\\ & = \# \left\{ 1 \leq l \neq m \leq N \ : \ \norm{x_l - x_m} \leq \frac{sN^{\alpha_1-\alpha_2}}{N^{\alpha_1}} \right\}\\
	& = 2sN^{\alpha_1-\alpha_2} \cdot N^{2-\alpha_1} + o(N^{2-\alpha_1}) = 2sN^{2-\alpha_2} + o(N^{2-\alpha_1})  
	\end{align*}
	which implies the assertion of the proposition because $\alpha_1 > \alpha_2$.
\end{proof}
The following is an immediate consequence of the proposition.
\begin{cor} Let $0 < \alpha \leq 1$ and $(X_i)$ be a sequence of i.i.d. random variables (sampled from the uniform distribution in $[0,1)$), then for all $s \geq 0$ we have $F_N^\alpha(s) \to 2s$, as $N \to \infty$ almost surely.
\end{cor}
The maximal parameter $\alpha$ such that a sequence $(x_n)$ has $\alpha$-pair correlations may be regarded as a measure how close a sequence is to having Poissonian pair correlations. The two main results of this paper concern all van der Corput sequences and the Kronecker sequence for the golden mean $\phi=\tfrac{1+\sqrt{5}}{2}$. Although van der Corput sequences and Kronecker sequences are not only uniformly distributed but even form classes of low-discrepancy sequences, it is well-known and easy to show that they do not have Poissonian pair correlations. In this paper we prove, that they fail as closely to have Poissonian pair correlations as possible in the sense of $\alpha$-pair correlations.
\begin{thm} \label{thm:vdC} Let $(x_n)$ be the van der Corput sequence in base $b$. The sequence $(x_n)$ has $\alpha$-pair correlations for all $0 < \alpha < 1$ but does not have Poissonian pair correlations.
\end{thm}
Recall that a \textbf{Kronecker sequence} is of the form $(z_n)_{n \geq 0} = (\left\{ nz \right\})_{n \geq 0}$ with $z \in \RR$. We restrict here to the Kronecker sequence for the golden mean since our proof is already rather technical in this case.
\begin{thm} \label{thm:kron} The sequence $(x_n) = \{ n\phi\}$ has $\alpha$-pair correlations for all $0 < \alpha < 1$ but does not have Poissonian pair correlations.
\end{thm}

\section{Van der Corput sequences}

Van der Corput sequences are classical examples of uniformly distributed sequences. They are defined as follows: for an integer $b \geq 2$ the $b$-ary representation of $n \in \NN$ is $n = \sum_{j=0}^\infty a_j(n) b^j$ with $a_j(n) \in \NN$. The radical-inverse function is defined by $g_b(n)=\sum_{j=0}^\infty a_j(n) b^{-j-1}$ for all $n \in \NN$. Finally, the \textbf{van der Corput sequence in base $b$} is given by $(x_n) = g_b(n)$. This section is dedicated to the proof of Theorem~\ref{thm:vdC}.

\begin{proof}[Proof of Theorem~\ref{thm:vdC}] The fact that $(x_n)$ does not have Poissonian pair correlations is e.g. an immediate consequence from the observation that the denominator of the first $b^{n}-1$ elements of $(x_n)$ is at most $b^n$ which implies $\norm{x_l-x_m} \geq \tfrac{1}{b^n}$ for all $l,m \leq b^n-1$.\\[12pt]
In order to facilitate the proof regarding the $\alpha$-pair correlations of van der Corput sequences, we add $0$ as zeroth element of $x_n$. Otherwise we would have to take into account that the sequence has an extra \textit{hole} at zero which would result in some extra factors $-1$ but, of course, not change the asymptotic. Let $\alpha > 0 $ and $s > 0$ be arbitrary but fixed. 
Let $N,n \in \NN$ with $b^{n-1} < N \leq b^{n}$ and define $\gamma_N := \tfrac{N^{\alpha}}{b^{n}}.$
Note that the greatest denominator which appears in $(x_n)_{n=0}^{N}$ is $b^{n}$. Without loss of generality we have $s > \gamma_N$ because $\gamma_N \to 0$ for $N \to \infty$. Then
\begin{align} \label{eq1}
	\norm{x_l - x_m} N^\alpha & = k \frac{N^\alpha}{b^{n}} \stackrel{!}{\leq} s
\end{align}
holds for some $k \in \NN$ if and only if $k \leq sb^{n}N^{-\alpha}$. For fixed $l$ there exist between $2\lfloor \tfrac{s}{b \gamma_N} \rfloor$ and $2\lfloor \tfrac{s}{\gamma_N} \rfloor$ (if $N = b^n$) many different $x_m$ satisfying \eqref{eq1}. For the simplest case, namely $N = b^{n}$, we get
\begin{align*} \# \left\{ 1 \leq l \neq m \leq N \ : \ \norm{x_l - x_m} \leq \frac{s}{N^\alpha} \right\} & = N \cdot 2\lfloor \tfrac{s}{\gamma_N}  \rfloor \leq 2N \frac{s}{\gamma_N} \\ & = 2s N N^{-\alpha} b^{n} = 2sN^{2-\alpha}
\end{align*} 
and 
\begin{align*} 2sN^{2-\alpha} - 2N & = N \frac{2s}{\gamma_N} - 2N  < N \cdot 2\lfloor \tfrac{s}{\gamma_N} \rfloor \\ & = \# \left\{ 1 \leq l \neq m \leq N \ : \ \norm{x_l - x_m} \leq \frac{s}{N^\alpha} \right\}.
\end{align*} 
Since $\alpha < 1$, the claim follows for $N = b^{n}$.\\[12pt]
If $b^{n} + j = N < b^{n+1}$, counting the number of points satisfying $\norm{x_l - x_m} \leq \tfrac{s}{N^{\alpha}}$ needs more effort. We ust the decomposition
\begin{align*} \# \left\{ 1 \leq l \neq \right.& m \leq N \ : \left. \ \norm{x_l - x_m} \leq \frac{s}{N^\alpha} \right\} \\ & = \# \underbrace{\left\{ 1 \leq l \neq m \leq b^n \ : \ \norm{x_l - x_m} \leq \frac{s}{N^\alpha} \right\}}_{=:A}\\ & + 2 \# \underbrace{\left\{ 1 \leq l \leq b^n, b^n < m \leq N \ : \ \norm{x_l - x_m} \leq \frac{s}{N^\alpha} \right\}}_{=:B}
\\ & + \# \underbrace{\left\{ b^n \leq l,m \leq N \ : \ \norm{x_l - x_m} \leq \frac{s}{N^\alpha} \right\}}_{=:C}
\end{align*}
The first summand can be calculated as in the case $N=b^n$ which leads to
\begin{align*} b^n \cdot 2\lfloor \tfrac{s}{\gamma_N} \rfloor - 2b^n \leq |A| = b^n \cdot 2\lfloor \tfrac{s}{\gamma_N}  \rfloor 
\end{align*}
Regarding the second summand, every point $x_m$ of the van der Corput sequence with $m > b^n$ lies between two points $x_i, x_j$ with $i,j < b^m$ and thus there exist between $2\lfloor \tfrac{s}{\gamma_N} \rfloor$ and $2(\lfloor \tfrac{s}{\gamma_N} \rfloor + 2)$ points $x_l$ within $B$ for fixed $x_m$. This yields
\begin{align*}
2(N-b^n) \lfloor \tfrac{s}{\gamma_N} \rfloor \leq |B| \leq 2 (N-b^n) (\lfloor \tfrac{s}{\gamma_N} \rfloor + 2).
\end{align*}
Note that the points $x_l$ relevant in the set $C$ form a displaced van der Corput sequence. Let us consider here the case $N = b^n + b^{n-1}$ because the general case then follows inductively. For $N= b^n + b^{n-1}$ we can apply the result for $b^{n-1}$ on the set $C$ to obtain 
\begin{align*}
b^{n-1} \cdot 2\lfloor \frac{s}{N^\alpha} b^{n-1} \rfloor - 2b^{n-1} \leq |C| = b^{n-1} \cdot 2\lfloor \frac{s}{N^{\alpha}} b^{n-1}  \rfloor 
\end{align*}
Summing up the three summands, dividing by $N^{2-\alpha}$ and taking the limit implies the claim for arbitrary $N$.
\end{proof}

\section{Golden mean Kronecker sequence} Before we come to the proof of Theorem~\ref{thm:kron} we collect here some necessary background.

\paragraph{Continued fractions.} Recall that every irrational number $z$ has a uniquely determined infinite continued fraction expansion
$$z = a_0 + 1/(a_1+1/(a_2+\ldots)) =: [a_0;a_1,a_2,\ldots],$$
where the $a_i$ are integers with $a_0 = \lfloor z \rfloor$ and $a_i \geq 1$ for all $i \geq 1$. The continued fraction algorithm can be compactly written as $t_0 = z, a_1 = \lfloor t_0 \rfloor$ and $t_{i+1} = 1 / \{t_i\}, a_{i+1} = \lfloor t_{i+1} \rfloor$ for $i \geq 1$. The sequence of \textbf{convergents} $(r_i)_{i \in \NN}$ of $z$ is defined by
$$r_i = [a_0;a_1;\ldots;a_i].$$
The convergents $r_i = p_i/q_i$ with $\gcd(p_{i},q_{i}) = 1$ are also directly given by the recurrence relation
\begin{align*}
& p_{-1} = 0, \qquad p_{0} = 1, \qquad p_{i} = a_ip_{i-1} + p_{i-2}, \quad i \geq 0\\
& q_{-1} = 1, \qquad q_{0} = 0, \qquad q_{i} = a_iq_{i-1} + q_{i-2}, \quad i \geq 0.
\end{align*}
The residue $z - r_i$ can be calculated precisely as
\begin{align} \label{eq2}
		z - \frac{p_n}{q_n}= \frac{(-1)^n}{q_n\left( \frac{1}{t_n} q_n + q_{n-1} \right)}.
\end{align}
Furthermore note that $\phi = [1;1,1,1,\ldots]$.

\paragraph{Three Gap Theorem.} An important ingredient for the proof is	Three Gap Theorem. It states that there are at most three distinct lengths of gaps if one places $n$ points on a circle, at angles of $z, 2z, 3z, \ldots nz$ from the starting point. The theorem was first proven in \cite{Sos58} and many proofs have been found since then. We present it here in a similar flavor as it is formulated in \cite{Wei18}. 

\begin{thm}[Three Gap Theorem] \label{thm:3gap} Let $z \in (0,1)$ be irrational with continued fraction expansion $z =[a_0;a_1;a_2;\ldots]$ and convergents $r_n = p_n/q_n$. Furthermore let $N \in \NN$ with $N \geq 2$ have Ostrowski representation
	$$N = \sum_{i=1}^{m} b_i q_i$$
	with coefficients $0 \leq b_i \leq a_i$ for all $i=1,\ldots, m$ and $b_{i-1} = 0$ if $b_i = a_i$. Then for $K_l := \norm{ \{ q_l z\}}$ the finite sequence $(\left\{ n z \right\})_{n=1,\ldots,N-1}$ has at most three different lengths of gaps, namely 
	\begin{align*}
	L_1 & = K_{m-1} - (b_m-1) K_{m}, \\
	L_2 & = K_{m}\\
	L_3 & = L_1 + L_2.
	\end{align*}
\end{thm}

\paragraph{Preparatory Results.} In the following $(x_n) = \{ n \phi \}$ denotes the Kronecker sequence of the golden mean. 

\begin{lem} \label{lem:growth} Let $s > 0, l \in \NN$ and $0 < \alpha < 1$. Then
	$$4s > \lim_{N \to \infty} \frac{1}{N^{2-\alpha}} \# \left\{ 1 \leq m \leq N \ : \ \norm{x_l - x_m} \leq \frac{s}{N^\alpha} \right\} > \frac{s}{2} $$
	holds, if the limit exists.
\end{lem}
\begin{proof} Let $q_h \leq N < q_{h+1}$. By the Three Gap Theorem, the large gaps of $(x_n)_{n=1}^N$ are of length $\norm{\{q_{h-2}z\}} < \tfrac{1}{q_{h-1}} < \tfrac{4}{q_h}$. Therefore we have
\begin{align*}
	\left\lfloor \frac{\frac{s}{N^\alpha}}{\{q_{h-2}z\}} \right\rfloor& > \frac{sN}{4N^\alpha} - 1.
\end{align*}
From this we deduce
\begin{align*}
	\frac{1}{N^{2-\alpha}} \# & \left\{ 1 \leq m \leq N \ : \ \norm{x_l - x_m} \leq \frac{s}{N^\alpha} \right\}\\
	& > \frac{1}{N^{2-\alpha}} \cdot N \cdot 2 \cdot \left(\frac{sN}{4N^\alpha} - 1 \right)
\end{align*}
which implies the right inequality. The left one follows in the same manner by recalling that the small gap length is $\norm{\{q_{h-1}z\}} > \tfrac{1}{q_h+q_{h-1}} > \tfrac{1}{2q_h} > \tfrac{1}{4N}$.
\end{proof}
Note that Lemma~\ref{lem:growth} is wrong in the case $\alpha = 1$. It implies that the numerator in $F_N^\alpha$ has the right order of convergence for $0 < \alpha < 1$ but it takes much more effort to show that the limit exists and calculate it explicitly. Define the finite sequence $(x_n^*)_{n=0}^N$ by sorting $(x_n)_{n=0}^N$ in ascending order. Let $J_k^N(x_n^*)$ denote the interval $(x_n^*;x_{n+k}^*]$. 

\begin{lem} \label{lem:multiplicities}
\begin{itemize}
	\item[(i)] Let $N = q_h$ and $k = q_m$ with $h \geq m$ and choose $n$ with $n+k \leq N$. Then $J_k^N(x_n^*)$ consists of either $q_{m-1}$ or $q_{m-1} + 1$ large and $q_{m-2}$ or $q_{m-2} -1$ small gaps respectively. 
	\item[(ii)] Let $q_{m+1} > k > q_m$ have Ostrowski representation
		$$k = \sum_{i=1}^{m} b_i q_i$$
	with coefficients $b_m = 1, b_i \in \left\{ 0 ,1\right\}$ for all $i=1,\ldots, m-1$ and $b_{i-1} = 0$ if $b_i = 1$. Furthermore choose $n \in \NN$ arbitrary and $N =q_h \geq n + q_{m+1}$. Then $J_k^N(x_n^*)$ consists of $g$ or $g+1$ large gaps and $k - g$ or $k-g-1$  small gaps respectively, where
	$$g = \sum_{i=1}^{m-1} b_i q_{i-1}$$
	\end{itemize}
\end{lem}

\begin{proof} (i) By the rotation symmetry of the Kronecker sequence we may without loss of generality assume $n=0$.  We prove the claim by induction on $h$. The case $h=1$ is clear. For $N=q_{h+1}$ and $k=q_{h+1}$ there is also nothing to prove. Hence let $k = q_{h-r}$ with $r \in \NN_0$. By induction hypothesis, the first $q_{h-r-1}$ elements of the sequence $(x_n^*)_{n=1}^{q_h}$ contain  without loss of generality $q_{h-r-2}$ large($q_h$) and $q_{h-r-3}$ small($q_h$) gaps.\footnote{By large($q_h$) and small($q_h$) we mean the large gaps and small gaps of $(x_n)_{n=0}^{q_h}$ respectively.} When passing from $(x_n^*)_{n=1}^{q_h}$ to $(x_n^*)_{n=1}^{q_h+1}$, each of the small($q_h)$ becomes large($q_{h+1})$ and each of the large($q_h$) gaps is divided into a large($q_{h+1}$) and a small($q_{h+1}$) gap, compare e.g. \cite{Wei18}. In total, we end up with $q_{h-r-2} + q_{h-r-3} = q_{h-r-1}$ or $q_{h-r-2} + q_{h-r-3} + 1= q_{h-r-1} + 1$ large gaps since we do not know whether the last gap is large or small.\\[12pt] 
	(ii) It suffices to prove the formula for the number of large gaps. Again we assume $n=0$. We know by (i) that $J^r:=J_{q_r}(x_0^*)$ consists of either $q_{r-1}$ or $q_{r-1}+1$ large gaps for $r = m, m+1$. At first, we consider two special cases:\\[12pt] 
	(a) If $b_{m-2} = 1$, then we look at $J:=J_{q_{m-2}}^N(x_{q_m}^*)$ and  $\widetilde{J}:= J_{q_{m-3}}^N(x_{q_m+q_{m-2}}^*) = J_{q_{m+1}-q_m-q_{m-2}}^N(x_{q_m+q_{m-2}}^*)$. Assertion (i) implies that $J$ consists of $q_{m-3}$ or $q_{m-3} + 1$ large gaps and $\widetilde{J}$ has $q_{m-4}$ or $q_{m-4} + 1$ large gaps. For $k = q_m + q_{m-2}$ the four conditions can only be fulfilled simultaneously if $J_k^N(x_0^*)$ consists of either $q_{m-1} + q_{m-3}$ or $q_{m-1} + q_{m-3} + 1$ large gaps. Inductively the claim follows for $s \leq m$ and $k=\sum_{i \leq s, i \, \textrm{even}} q_{m-i}$.\\[12pt] 
	(b) If $b_{m-2} = 0$ and $b_{m-3}=1$, then note that $q_{m+1}-q_m-q_{m-3} = q_{m-2}$ and we can deduce the claim from (i) in the same manner as in (a).\\[12pt] In general, let $s\geq 3$ be minimal with $b_{m-s} = 1$ and $b_{m-1} = \ldots = b_{m-(s+1)} = 0$. If $s$ is odd, then $q_{m+1} - q_m - q_s = q_{m-2}+q_{m-4}+ \ldots + q_{m-s+1}$ and if $s$ is even, then $q_{m+1} - q_m - q_s = q_{m-2} + q_{m-4} + \ldots + q_{m-s+2} - q_{m-s-1}$. The two cases correspond to (a) and (b) respectively. 
\end{proof}

The limit of the fraction of the two expressions $k$ and $g$ from Lemma~\ref{lem:multiplicities} is calculated next.
\begin{lem} \label{lem:limitquotient} The following holds:
	$$\lim_{N \to \infty} \frac{\sum b_iq_i}{\sum b_i q_{i-1}} = \phi.$$
\end{lem}
\begin{proof} It follows from $q_i = q_{i-1}+q_{i-2}$ that
	$$
		\frac{\sum b_iq_i}{\sum b_i q_{i-1}}  = \frac{\sum b_i(q_{i-1}+q_{i-2})}{\sum b_i q_{i-1}} = 1 + \frac{\sum b_iq_{i-2}}{\sum b_i q_{i-1}}.
	$$
	Thus the limit fulfills the equation
	$$z = 1 + \frac{1}{z}$$
	which has the unique positive solution $\phi$.
\end{proof}

From Lemma~\ref{lem:multiplicities} and Lemma~\ref{lem:limitquotient} we see that in the limit there are $\phi$ times as many long gaps as short ones. Another limit will appear in the proof of Theorem~\ref{thm:kron}. For clarity of presentation we also calculate this limit as a separate lemma.

\begin{lem} \label{lem:limit1} We have
	$$\lim_{h \to \infty} \left| \left(1+\frac{1}{\phi^2} \right) \cdot \norm{\{ q_{h-1} \phi \}} \cdot q_h \right| = 1$$
\end{lem}

\begin{proof}
	Note that for $\phi$ we have $\frac{1}{|{t_l}|} = \phi$ and $p_l = q_{h-1}$. Hence it follows from \eqref{eq2} that
	$$\left| q_{h-1}\phi - q_{h-2} \right| = \frac{1}{\phi q_{h-1}+q_{h-2}}.$$
	Multiplying both sides by $q_h$ leads to the expression
	\begin{align*}
		q_h \left| q_{h-1}\phi - q_{h-2} \right| & = \frac{q_h}{\phi q_{h-1} + q_{h-2}}\\
		& = \frac{1}{\phi \frac{q_{h-1}}{q_h} + \frac{q_{h-2}}{q_h}}, 
	\end{align*}
	which converges to 
	$$\frac{1}{1+\frac{1}{\phi^2}}$$
	for $h \to \infty$. This implies the claim.
\end{proof}

\paragraph{Proof of Theorem \ref{thm:kron}} Finally, we come to the proof of our second main result.

\begin{proof}[Proof of Theorem \ref{thm:kron}] At first, we give here a short proof of the well-known fact that $\left\{ n\phi \right\}$ does not have Poissionian pair correlations: Consider the sequence $(q_h)_{h_\in \NN}$. According to Three Gap Theorem the minimal gap length of the Kronecker sequence is $\{ q_{h-1} \phi \}$. From the theory of continued fractions we know that
	$$\norm{\left\{q_{h-1} \phi \right\}} = \left| q_{h-1} \phi - p_{h-1} \right| > \frac{1}{q_h+q_{h-1}} > \frac{1}{2q_h}.$$
Therefore, $F_{q_h}^1(\tfrac{1}{2}) = 0$ for all $h$ and the sequence $(x_n)$ cannot have Poissonian pair correlations.\\[12pt]
	Next we come to the claim that $(x_n)$ has $\alpha$-pair correlations. Let $0 < \alpha < 1$ be arbitrary. We split up the proof into two parts.\\[12pt]
	\textit{Step 1: Calculate the limit}
		$$\lim_{h \to \infty} \frac{1}{q_h^{2-\alpha}} \# \left\{ 1 \leq l \neq m \leq q_h \ : \ \norm{x_l - x_m} \leq \frac{s}{q_h^\alpha} \right\},$$
	i.e. we consider $N = q_h$ at first. In this case, Three Gap Theorem implies that there are only two different gap lengths, namely $\norm{\{q_{h-1}\phi\}}$ and $\norm{\{q_{h-1} \phi\}} + \norm{\{q_{h}\phi\}}$. We fix some $x_n^*$. By Lemma~\ref{lem:multiplicities} the length of $J_k^N({x_n^*})$ is for $k = \sum_{i=1}^{m} b_i q_i$ and $g = \sum_{i=1}^{m-1} b_i q_{i-1}$ either
	\begin{align*}
	|J_k^N(x_n^*)| & = g \cdot (\norm{\{q_{h-1} \phi\}} + \norm{\{q_{h}\phi\}}) + (k-g) \cdot \norm{\{q_{h-1}\phi\}} \\
	&  = k \cdot \norm{\{q_{h-1}\phi\}} + g \cdot \norm{\{q_{h}\phi\}}
	\end{align*}
	or 
	\begin{align*}
	|J_k^N(x_n^*)| & = (g-1) \cdot (\{q_{h-1} \phi\} + \{q_{h}\phi\}) + (k-g-1) \cdot \{q_{h-1}\phi\} \\
	&= k \cdot \norm{\{q_{h-1}\phi\}} + (g-1) \cdot \norm{\{q_{h}\phi\}}
	\end{align*}
	Without loss of generality we may assume that $J_k^N(x_n^*) = k \cdot \norm{\{q_{h-1}\phi\}} + g \cdot \norm{\{q_{h}\phi\}}$ because a single point less does not have an influence on the asymptotic. Now let $\varepsilon > 0$ be arbitrarily small and choose $h$ big enough such that simultaneously $|\tfrac{g}{k} - \tfrac{1}{\phi}| < \varepsilon$ and $|\norm{\{ q_h \phi \}} - \tfrac{1}{\phi} \norm{\{ q_{h-1} \phi \}}| < \varepsilon$ hold. Given an interval $I = (x_n^*,x_n^*+\tfrac{s}{q_h^\alpha}]$, the number of points lying in it can thus be calculated as
	\begin{align*}
		k & = \frac{\frac{s}{q_h^\alpha}}{\left(1+\frac{1}{\phi^2} \right)\norm{\{q_{h-1}\phi\}} + O(\varepsilon)}
	\end{align*}
	Note that this equality is true without any additional condition on $x_n^*$ if we glue the endpoints of $[0,1)$. After increasing $h$ if necessary, Lemma~\ref{lem:limit1} yields 
	\begin{align} \label{eq3}
		k & = s q_h^{1-\alpha} + O(\varepsilon).
	\end{align}
	This enables us to calculate
	\begin{align*}
		\lim_{h \to \infty} \frac{1}{q_h^{2-\alpha}} & \# \left\{ 1 \leq l \neq m \leq q_h \ : \ \norm{x_l - x_m} \leq \frac{s}{q_h^\alpha} \right\}\\ 
		&=  \lim_{h \to \infty} \frac{1}{q_h^{2-\alpha}} q_h \cdot \# \left\{ 1 \leq m \leq q_h \ : \ \norm{x_n - x_m} \leq \frac{s}{q_h^\alpha} \right\}\\
		&\stackrel{\eqref{eq3}}{=} \lim_{h \to \infty} \frac{1}{q_h^{2-\alpha}} q_h \cdot 2 \cdot (sq_h^{1-\alpha} + O(\varepsilon))\\
		&=2s.
	\end{align*}
	\textit{Step 2: $N$ arbitrary}\\[12pt]
	Let $q_{h+1} > N > q_h$ have Ostrowski decomposition
	$$N = \sum_{i=1}^{h} b_i q_i.$$
	Without loss of generality we assume that $N = q_h + q_{h-p}$ for some $p > 0$ because the general case then easily follows inductively. Similarly as for the van der Corput sequence, we decompose 
	\begin{align*} \# \left\{ 1 \leq l \neq \right.& m \leq N \ : \left. \ \norm{x_l - x_m} \leq \frac{s}{N^\alpha} \right\}
	\end{align*}
	into three subsets, namely
	\begin{align*} \# \left\{ 1 \leq l \neq \right.& m \leq N \ : \left. \ \norm{x_l - x_m} \leq \frac{s}{N^\alpha} \right\} \\ & = \# \underbrace{\left\{ 1 \leq l \neq m \leq q_m \ : \ \norm{x_l - x_m} \leq \frac{s}{N^\alpha} \right\}}_{=:A}\\ & + 2 \# \underbrace{\left\{ 1 \leq l \leq q_h, q_h < m \leq N \ : \ \norm{x_l - x_m} \leq \frac{s}{N^\alpha} \right\}}_{=:B}
	\\ & + \# \underbrace{\left\{ q_h \leq l,m \leq N \ : \ \norm{x_l - x_m} \leq \frac{s}{N^\alpha} \right\}}_{=:C}
	\end{align*}
	Let $A^*$ be the set $A$ with $s$ replaced by $s (\tfrac{N}{q_h})^\alpha$. By the first step we can then choose $h$ big enough such that $|A^*| = 2sq_h^{2-\alpha}  + o(N^{2-\alpha})$. Going back to $A$ this implies $|A| = 2sq_h^2N^{-\alpha} + o(N^{2-\alpha})$. For the calculation of $|C|$ note that $x_{n+q_h}$ is a displaced Kronecker sequence. Likewise as for $A$ we hence get $|C| = 2sq_{h-p}^2N^{-\alpha} + o(N^{2-\alpha})$. Finally we come to the set $B$. It concerns the interaction of the \textit{new} points $x_m$, i.e. $m > q_h$, with the \textit{old} ones $x_l$, i.e. $l \leq q_h$. Note that $B$ contains $N-q_h=q_{h-p}$ new points and that each new point splits a long gap into a short gap and a medium size gap, compare \cite{Wei18}. Again by the first step and a scaling argument as for $A$ and $C$ we thus know that a single new point contributes $2sq_hN^{-\alpha} + o(N^{1-\alpha})$ to $|B|$. In total, we get $2|B| = 4sq_{h-p}q_hN^{-\alpha} + o(N^{2-\alpha})$. That puts us into the position to calculate the limit
	\begin{align*}
	\lim_{N \to \infty} \frac{1}{N^{2-\alpha}} & \# \left\{ 1 \leq l \neq m \leq N \ : \ \norm{x_l - x_m} \leq \frac{s}{N^\alpha} \right\}\\
	& = \frac{N^{-\alpha} \cdot 2s \cdot (q_h^2+2q_hq_{h-p}+q_{h-p}^2) + o(N^{2-\alpha})}{N^{-\alpha} \cdot (q_h+q_{h-p})^2}\\& = 2s,
	\end{align*}
	which finishes the proof.
\end{proof}

\textsc{Christian Wei\ss, Hochschule Ruhr West, Duisburger Str. 100, D-45479 M\"ulheim an der Ruhr}\\
\textit{E-mail address:} \texttt{christian.weiss@hs-ruhrwest.de}\\[12pt]
\textsc{Thomas Skill, Hochschule Bochum, Lennershofstraße 140, D-44801 Bochum}\\
\textit{E-mail address:} \texttt{thomas.skill@hs-bochum.de}

\end{document}